\documentclass[12pt]{article}
\usepackage[utf8]{inputenc}
\usepackage[english]{babel}

\usepackage{enumerate}
\usepackage[usenames,dvipsnames,table]{xcolor}
\usepackage{color}
\usepackage{amsmath,amsthm,amssymb}
\usepackage{amsfonts}
\usepackage[width=16.00cm, height=24.00cm]{geometry}
	 
\usepackage{fourier}

\usepackage[colorlinks=true,linkcolor=colorref,citecolor=colorcita,urlcolor=colorweb]{hyperref}
\definecolor{colorcita}{RGB}{21,86,130}
\definecolor{colorref}{RGB}{5,10,177}
\definecolor{colorweb}{RGB}{177,6,38}

\DeclareMathOperator{\co}{co}
\DeclareMathOperator{\id}{id}
\DeclareMathOperator{\spn}{span}

\newcommand{\N}{\mathbb{N}}
\newcommand{\C}{\mathbb{C}}

\newtheorem{teo}{Theorem}[section]
\newtheorem{lema}[teo]{Lemma}
\newtheorem{coro}[teo]{Corollary}
\newtheorem{prop}[teo]{Proposition}
\theoremstyle{definition}
\newtheorem{example}[teo]{Example}
\newtheorem{nota}[teo]{Remark}

\title{Mean ergodic composition operators on spaces of holomorphic functions on a Banach space}

\author{David Jornet \and Daniel Santacreu \and Pablo Sevilla-Peris}

\date{}

\begin{document}
\maketitle

\begin{abstract}
We study mean ergodic composition operators on infinite dimensional spaces of holomorphic functions of different types when defined on the unit ball of a Banach or a Hilbert space: that of all holomorphic functions, that of holomorphic functions of bounded type and that of bounded holomorphic functions. Several examples in the different settings are given.
\end{abstract}

\footnotetext[0]{\textit{Keywords:} Holomorphic function on Banach space; Composition operator; Power bounded; Mean ergodic; Bounded type\\
Insitut Universitari de Matem\`atica Pura i Aplicada. Universitat Polit\`ecnica de Val\`encia. cmno Vera sn, 46022. Val\`encia. Spain.\\
djornet@mat.upv.es, dasanfe5@posgrado.upv.es,  psevilla@mat.upv.es}

\section{Introduction}

 If $X$ and $Y$ are Banach spaces and $U \subseteq X$ is open, then a function $f : U \to Y$ is \textit{holomorphic} if it is Fr\'echet differentiable at every point of $U$. If the open unit ball $B_X$ of $X$ satisfies $B_X\subseteq U$ and $\varphi : B_{X} \to B_{X}$ is a holomorphic self-map on $B_X$, the associated composition  operator is defined by $C_{\varphi} (f) := f \circ \varphi$. The function $\varphi$ is called \emph{symbol} of the composition operator. When $Y=\C$ and $U=B_X$, the space of holomorphic functions $f:B_X\to \C$ is simply denoted by $H(B_X)$. Our aim is to study the power boundedness and (uniform) mean ergodicity of the composition operator $C_\varphi:H(B_X)\to H(B_X)$  in terms of the properties of the symbol $\varphi$ when $H(B_X)$ is equipped with its natural topology, the compact-open topology, and also when $H(B_X)$ is replaced by the space of holomorphic functions of bounded type $H_b(B_X)$ or that of bounded holomorphic functions  $H^\infty(B_{X})$. We study also the case when $X$ is a Hilbert space for each of the settings considered above.

Several authors have studied different properties of composition operators on spaces of holomorphic functions on the unit ball of a Banach space. See, for instance, \cite{aron1997compact,galindo2010fredholm,galindo2000weakly,garcia2004composition} and the references therein. However, it seems that there is no previous literature about the dynamics of such operators.  The present work can be considered a sequel of \cite{JoSaSe20} by the same authors, where we study some dynamical properties (especially mean ergodicity) of composition operators in spaces of homogeneous polynomials. As in \cite{JoSaSe20}, the   motivation and inspiration of our investigation comes from several previous works, as \cite{BoDo2011A}, where the authors characterise those composition operators $C_{\varphi}:H(U)\to H(U)$ which are power bounded, where  $H(U)$ is the space of holomorphic functions on a connected domain of holomorphy $U$ of $\C^{d}$. It was proved in \cite{BoDo2011A} that $C_{\varphi}$ is power bounded if and only if it is (uniformly) mean ergodic, and this happens if and only if the symbol $\varphi$ has stable orbits. On the other hand, if the domain is the unit disc, it was characterised in \cite{BGJJ2016m} when $C_{\varphi}$ is mean ergodic or uniformly mean ergodic on the disc algebra or on the space of bounded holomorphic functions in terms of the asymptotic behaviour of the symbol. Power boundedness and (uniform) mean ergodicity of weighted composition operators on the space of holomorphic functions on the unit disc was analysed in \cite{BGJJ2016mw} in terms of the symbol and the  multiplier. In \cite{K2019p}  power boundedness and  mean ergodicity for (weighted) composition operators on function spaces defined by local properties was studied in a very general framework which extends previous work. In particular, it permits to characterise (uniform) mean ergodicity for composition operators on a large class of function spaces which are Fr\'echet-Montel spaces when equipped with the compact-open topology. Here, the results of \cite{K2019p} do not apply since $H(B_X)$, $H_b(B_X)$ or $H^\infty(B_{X})$ are not  Fr\'echet-Montel spaces. Other recent contributions to this topic can be found in \cite{Keshavarzi}, where mean ergodicity of 
composition operators on the space of bounded holomorphic functions on the $n$-dimensional Euclidean ball is studied, and in \cite{JA}, where the authors consider composition operators on weighted spaces of holomorphic functions on the disc.

The paper is organised as follows. In Section~\ref{preliminares} we give some basic definitions and fix the notation used throughout the paper. Moreover, we recall a specific result for Hilbert spaces which is useful along the text. In Section~\ref{stable}  we analyse some properties of stable and $B_X$-stable orbits. In Section~\ref{H} we study the mean ergodicity of the composition operator in the space of holomorphic functions on the unit ball of a Banach space. In Section~\ref{Hb} we consider the same problem for holomorphic functions of bounded type, while in Section~\ref{Hi} we consider the space of bounded holomorphic functions. In each section we treat the Hilbert-space case also.

\section{Preliminaries} \label{preliminares}

All along this paper $E$ will always denote a locally convex Hausdorff space. The set of continuous seminorms on $E$ is denoted by $\Gamma_{E}$ and $L(E)$ is the space of continuous linear maps $T: E \to E$. We denote $T^{0}= \id$ (the identity), $T^{1} = T$ and, for $n \in \mathbb{N}$, we write $T^{n} = T^{n-1} \circ T$ (that is, the $n$-th composition of $T$ with itself). With 
this notation, the $n$-the Ces\`aro mean of the sequence $(T^k)_{k=0}^\infty$ is defined as
\[ 
T_{[n]}:= \frac{1}{n}\sum_{k=0}^{n-1} T^k \,. 
\]
The operator $T$ is said to be
\textit{power bounded} if $\{T^{n} \colon n \geq 0 \}$ is equicontinuous. It is called \textit{mean ergodic} if there is $L\in L(E)$ such that $(T_{[n]} x)_{n}$ is convergent (in $E$) to $Lx$ for every $x \in E$. 
It is \emph{uniformly mean ergodic} if  $(T_{[n]})_{n}$ converges uniformly on the bounded subsets of $E$ (we will refer to the topology so defined as \textit{the topology of bounded convergence of $L(E)$}). Finally, we say that $T$ is \emph{topologizable} if for each $q\in \Gamma_E$ there exist a sequence $(a_n)_{n\in\N}$ of positive numbers and $p\in \Gamma_E$ such that 
\begin{equation} \label{topologizable}
q(a_nT^n x)\leq p(x), 
\end{equation}
for all $x\in E$ and all $n\in\N$ (see  \cite{bonet2007topologizable,zelazko2007topological}).

Also $X$ will always denote a Banach space and $H$ a Hilbert space. We write $E'$, $X'$ and $ H'$ for the corresponding dual spaces. The set $B_X$ is the open unit ball in $X$. On $B_{H}$ (recall that $H$ is a Hilbert space) there is a  group of automorphisms that, in some sense, plays the role of M\"obius transforms in the unit disc. We give here the definition and a basic property that we  use later.

From \cite[Proposition~1]{AutoInHilbert} we know that, given $a \in B_{H}$, the linear operator $\gamma_{a}: H \to  H$ defined by
\[
\gamma_{a}(x):=\frac{1}{1+v(a)}a\langle x,a\rangle + v(a)x \,,
\]
where $v(a)=\sqrt{1-\Vert a\Vert^2}$, satisfies $\Vert\gamma_{a}(x)\Vert \leq \Vert x\Vert$ for all $x\in H$ and $\gamma_{a}(a)=a$. Once we have this, for each $a \in B_{H}$ we can define  an automorphism $\alpha_{a} : B_{H} \to B_{H}$ by doing
\begin{equation}\label{alpha}
\alpha_a(x)=\gamma_{a} \Big(\frac{a-x}{1-\langle x, a\rangle} \Big) \,.
\end{equation}
This satisfies $\alpha_a(0)=a$, $\alpha_a(a)=0$, and $\alpha_a^{-1}=\alpha_a$ (the first two  follow by direct computation, and the third one proceeding as in  \cite[Proposition~1]{AutoInHilbert}). The following result follows from \cite[(9')]{AutoInHilbert}; we include a proof for the sake of completeness.

\begin{lema}\label{auto rB}
For each $0<r<1$ there is $0<\rho<1$ such that 
\begin{equation} \label{castillo}
	\alpha_a (rB_{H}) \subseteq \rho B_{H},
\end{equation}
for every $a\in rB_H$.
\end{lema}
\begin{proof}
For $x\in B_H$ with $\|x\|<r$, we put $y:=\alpha_a(x)$. Straightforward (though long) computation (see \cite[(2)]{AutoInHilbert}) yields
\[
1-\|y\|^2=\frac{(1-\|a\|^2)(1-\|x\|^2)}{|1-\langle x, a\rangle|^2}.
\]
Since 
\[
|1-\langle x, a\rangle|^2\le (1+\|x\|\|a\|)^2\le (1+r)^2,
\]
we deduce $1-\|y\|^2\ge (1-r^2)^2 (1+r)^{-2}= (1-r)^2$, which gives the conclusion for $\rho:=\sqrt{1-(1-r)^2}$.
\end{proof}

A mapping $P:X \to Y$ between two Banach spaces $X$ and $Y$ is a (continuous) \emph{$m$-homogeneous polynomial} if there is a continuous $m$-linear mapping $L : X \times \cdots \times X \to Y$ so that $P(x) = L(x, \ldots , x)$ for every $x \in X$. We write $\mathcal{P} (^{m} X)$ for the space of all $m$-homogeneous polynomials $P:X\to \C$, which endowed with the norm $\Vert P \Vert = \sup_{\Vert x \Vert \leq 1} \Vert P(x) \Vert$ is a Banach space.

We refer the reader to \cite{Me98,MeVo97} for general theory of functional analysis and Banach space theory, to \cite{libro_2019, Di99, mujica} for the theory of holomorphic functions on Banach spaces and to \cite{BaMa09,GEPe11} for topics related with linear dynamics.

\section{Stable and $B_{X}$-stable orbits} \label{stable}

Given an open set $U \subseteq X$, following \cite{BoDo2011A} a self map $f : U \to U$ is said to have \emph{stable orbits} if for every compact subset $K$ of $U$ there is a compact subset $L\subset U$ such that $f^n(K) \subseteq L$ for every $n\in\N$ or equivalently, if  $\overline{\bigcup_{n=0}^{\infty} f^{n}(K)}$ is compact in $U$ for every compact set $K \subseteq U$. This property was already used in
\cite{BoDo2011A} or \cite{BGJJ2016mw} to characterise power boundedness and/or mean ergodicity of weighted composition operators. 

We introduce now a sort of `bounded type' counterpart. A set $A \subseteq U$ is $U$-bounded if it is bounded and has positive distance to the boundary of $U$ (whenever $U=X$, the notions of `bounded' and  `$X$-bounded' coincide). Then we say that $f$ has \emph{$U$-stable orbits} if for every $U$-bounded set $A\subset U$ there is a $U$-bounded set $L\subset U$ such that $f^n(A) \subseteq L$ for every $n\in\N$ (equivalently, $\bigcup_{n=0}^{\infty} f^{n}(A)$ is $U$-bounded for every $U$-bounded set $A \subseteq U$).

\begin{nota}\label{point-so}
The orbit $\{f^{n} (x) \colon n \in \mathbb{N}  \}$ of each point $x \in U$ is relatively compact if $f$ has stable orbits and $U$-bounded if $f$ has $U$-stable orbits.
\end{nota}

The notion of a function having $B_{X}$-stable orbits (we only deal with the case $U=B_{X}$) seems to be new. However, it is not hard to find functions with this property. In fact,  the following well known version of the Schwarz lemma gives immediate examples. 

\begin{lema} \label{schwarz}
Let  $\varphi : B_{X} \to B_{X}$ be holomorphic so that 	$\varphi(0)=0$. Then
$\Vert \varphi(x) \Vert \leq \Vert x \Vert$ for every $x \in B_{X}$.
\end{lema}

\begin{proof}
	It is enough to apply the classical Schwarz lemma to the family of functions $$\Big\{ [\lambda \in \mathbb{D} \mapsto x^{*} \big( \varphi( \lambda x/\Vert x \Vert ) \big ) ]  \colon x^{*} \in X^{*} , \, \Vert x^{*} \Vert \leq 1, \, 0 < \Vert x \Vert< 1 \Big\}.$$
\end{proof}

\begin{prop}\label{fix0}
Let $\varphi: B_X \to  B_X$ be a holomorphic mapping such that $\varphi(0)=0$, then $\varphi$ has $B_X$-stable orbits.
\end{prop}
\begin{proof}
Lemma~\ref{schwarz}  clearly implies $\Vert\varphi^n(x)\Vert \leq \Vert x\Vert$ for all $n\in \N$ and all $x\in B_X$ and, therefore, for each $0<r<1$ we have 
\[ 
\varphi^n(rB_X)\subseteq rB_X, 
\]
for all $n\in \N$. This gives the claim.	
\end{proof}

As a consequence, every continuous homogeneous polynomial $P: X \to X$ (in particular every linear operator) with $\Vert P \Vert \leq 1$ has $B_{X}$-stable orbits. 

\begin{example} \label{shifts}
If $X$ is either $c_{0}$ or $\ell_{p}$ with $1 \leq p \leq \infty$ we consider the forward and backward shifts operators $F,B: X \to X$ defined as
\begin{equation} \label{def:shifts}
F(x_1,x_2,\dots)=(0,x_1,x_2,\dots) \,\text{ and } \,  B(x_1,x_2,\dots)=(x_2,x_3,\dots) \,.
\end{equation}
Both are linear and clearly have norm less or equal $1$, hence have $B_{X}$-stable orbits. It is not difficult to see that $B$ has stable orbits (just using the characterisation of compact sets in $c_{0}$ or in $\ell_{p}$; see for instance \cite[p.~6]{diestel1984sequences}). For the forward shift, however, we have that the set
\[ 
\left\{F^n\bigg(\frac{e_1}{2}\bigg):n\in\N \right\}= \left\{\frac{e_n}{2}:n>1 \right\}
\]
is not relatively compact and, by Remark~\ref{point-so}, $F$ does not have stable orbits.

We may also consider the mapping $\phi : B_X \to B_X$ defined as
\[ 
\phi(x_1,x_2,\dots)=\big( \tfrac{x_1 +1 }{2},0,0,\dots \big) \,.
\]
Note that $\phi^n(0)=\big( \sum_{i=1}^{n}\frac{1}{2^i},0,0,\dots \big)$ and, therefore, 
\[ 
\lim_{n\to\infty} \Vert \phi^n(0) \Vert = \lim_{n\to\infty}\sum_{i=1}^{n}\frac{1}{2^i}=1. 
\]
Hence $\phi$ has neither stable nor $B_{c_0}$-stable orbits.
\end{example}

We do not know so far whether or not having stable orbits implies having $B_{X}$-stable orbit. However, if $T:X \to X$ is continuous and linear and has stable orbits, then it is power bounded (because $\{T^{n} x  \}_{n}$ is bounded for every $x\in X$), and a simple computation shows that, then $T$ has $X$-stable orbits. 

\subsection{The Hilbert-space case}

If $H$ is a Hilbert space, for each $a \in B_{H}$ the automorphism $\alpha_{a} : B_{H} \to B_{H}$ defined in \eqref{alpha} satisfies  $\alpha_{a}^{-1}=\alpha_{a}$.
Hence
\[
\bigcup_{n=0}^{\infty} \alpha_{a}^{n} (A) = A \cup \alpha_{a} (A) \,,
\]
for every $A \subseteq B_{H}$. If $A$ is compact, $\alpha_{a}(A)$ is again compact, and if $A$ is $B_{H}$-bounded, by Lemma~\ref{auto rB}, so also is $\alpha_{a}(A)$. This shows that $\alpha_{a}$ has both stable and $B_{H}$-stable orbits.

Using these automorphisms, in the case of Hilbert spaces we can extend Proposition~\ref{fix0} showing that every holomorphic function with a fixed point has $B_{H}$-stable orbits.

\begin{lema} \label{compos stable}
If $\varphi: B_H\rightarrow B_H$ has stable orbits (respectively $B_{H}$-stable orbits), then the mapping $\psi= \alpha_a\circ \varphi \circ \alpha_a$ has stable orbits (respectively $B_{H}$-stable orbits) for every $a \in B_{H}$.
\end{lema}
\begin{proof}
If $K \subseteq B_{H}$ is compact, then $\alpha_{a} (K)$ is compact and, having $\varphi$ stable orbits, we can find a compact set $L \subseteq B_{H}$ so that $\varphi^{n} (\alpha_{a} (K)) \subseteq L$ for each $n\in\N$. Then $\alpha_{a} (L) \subseteq B_{H}$ is compact and $\alpha_{a} \big(\varphi^{n} (\alpha_{a} (K))  \big) \subseteq \alpha_{a} (L)$. Since  $\psi^{n} = \alpha_{a} \circ \varphi^{n} \circ \alpha_{a}$ (because $\alpha_{a}^{2} = \id$), $\psi$ has stable orbits. 

The argument if $\varphi$ has $B_{H}$-stable orbits is exactly the same, using that by Lemma~\ref{auto rB} $\alpha_{a}(A)$ is $B_{H}$-bounded for every $B_{H}$-bounded $A$.
\end{proof}

\begin{prop}\label{phi-bar-Bso}
Let $\varphi: B_H\rightarrow B_H$ be a holomorphic mapping with a fixed point. % such that there exists $a\in B_H$ with $\varphi(a)=a$. 
Then $\varphi$ has $B_H$-stable orbits.
\end{prop}
\begin{proof}
Take $a\in B_H$ with $\varphi(a)=a$.
The holomorphic function $\overline{\varphi}= \alpha_a\circ \varphi \circ \alpha_a : B_{H} \to B_{H}$ satisfies $\overline{\varphi}(0)= \alpha_a (\varphi ( \alpha_a(0)))=\alpha_a (\varphi (a))=\alpha_a (a)=0$. Then, by Proposition~\ref{fix0} the function $\overline{\varphi}$ has $B_H$-stable orbits, and Lemma~\ref{compos stable} gives the conclusion.
\end{proof}

\section{The space of holomorphic functions} \label{H}

Given a Banach space $X$, we define $H(B_{X})$ as the space of all holomorphic functions $f : B_{X} \to \mathbb{C}$, endowed with the topology $\tau_{0}$ of uniform convergence on compact sets. This is a locally convex Hausdorff space.

\begin{nota} \label{hubbard}
If $\varphi:B_{X} \to B_{X}$ is holomorphic, then the composition operator $C_{\varphi} : H(B_{X}) \to H(B_{X})$ is clearly well defined (and continuous). On the other hand, if $C_{\varphi}$ is well defined, then $x' \circ \varphi$ is holomorphic for every $x' \in X'$ and, by Dunford's theorem (see e.g. \cite[Theorem~15.45]{libro_2019}), $\varphi$ is holomorphic. Then, there is no restriction to assume that $\varphi$ is holomorphic.
\end{nota}

\begin{nota}\label{santsaens}
Suppose that the composition operator $C_{\varphi}$ is topologizable.
Given any $f \in H(B_{X})$, a straightforward computation using \eqref{topologizable} with $f^{m}$ (for $m \in \mathbb{N}$) and taking the $m$-root shows that for every compact set $K \subseteq B_{X}$ there is some compact $L \subseteq B_{X}$ so that
\[
\sup_{x \in K} \vert f( \varphi^{n} (x) ) \vert \leq \frac{1}{a_{n}^{1/m}} \sup_{x \in L} \vert f(x) \vert \,.
\]
Letting $m \to \infty$ yields
\[
\sup_{x \in K} \vert f( \varphi^{n} (x) ) \vert \leq \sup_{x \in L} \vert f(x) \vert \,,
\]
and in particular $C_{\varphi}$ is power bounded. Our aim now is to show that in fact this implication can be reversed and characterised in terms of the symbol.

\end{nota}

Following \cite{vieira} and \cite{carandomuro}, given a family $\mathcal{F}$ of $\mathbb{C}$-valued holomorphic functions defined on an open set $U$, the $\mathcal{F}$-hull of $A \subseteq U$ is denoted 
\begin{equation} \label{Fhull}
\widehat{A}_{\mathcal{F}}= \{ x\in U : |f(x)|\leq \sup_{y\in A} |f(y)|, \text{ for all } f \in \mathcal{F}  \} \,.
\end{equation}

Stable orbits of the symbol is the property that characterises the power boundedness of the composition operator.

\begin{teo}\label{PBinH}
Let $\varphi:B_{X} \to B_{X}$ be holomorphic. The following assertions are equivalent:
\begin{enumerate}%[(a)]
	\item\label{PBinH1} $\varphi$ has stable orbits on $B_{X}$.
	\item \label{PBinH2} $C_\varphi :H(B_{X}) \to H(B_{X})$ is power bounded.
	\item \label{PBinH3} $\big(\frac{1}{n} C_{\varphi}^{n} \big)_{n}$ is equicontinuous in $L(H(B_{X}))$.
	\item \label{PBinH4} $C_{\varphi} :H(B_{X}) \to H(B_{X})$ is topologizable.		
\end{enumerate}
\end{teo}
\begin{proof}
\ref{PBinH1} $\Rightarrow$ \ref{PBinH2} If $\varphi$ has stable orbits, given a compact set $K \subseteq B_{X}$  there is a compact set $L\subseteq B_{X}$ such that $\varphi^n(K) \subseteq L$ for every $n\in\N$. Hence
\[
\sup_{x\in K} |C_\varphi^n(f)(x)|= \sup_{x\in K} |f(\varphi^n(x))| \leq \sup_{x\in L}|f(x)|,
\]
for all $f\in H(B_{X})$ and $n\in\N$.
So the sequence $(C_\varphi^n)_n$ is equicontinuous and $C_{\varphi}$ is power bounded.

\ref{PBinH2} $\Rightarrow$ \ref{PBinH3}
Suppose now that  $C_{\varphi}$ is power bounded, then for each compact set $K \subseteq B_{X}$ we can find 
$c>0$ and a compact set $L \subseteq B$ so that
\[
\sup_{x\in K} |C_\varphi^n(f)(x)|\leq c \sup_{x\in L} |f(x)|
\]
for every $f \in H(B_{X})$ and $n \in \mathbb{N}$. This obviously implies
\[
\sup_{x\in K} \big| \tfrac{1}{n} C_\varphi^n(f)(x) \big| \leq c \sup_{x\in L} |f(x)|
\]
for every $f$ and $n$, and $\big(\frac{1}{n} C_{\varphi}^{n} \big)_{n}$ is equicontinuous.

\ref{PBinH3} $\Rightarrow$ \ref{PBinH4} follows just taking $a_{n}= \frac{1}{cn}$ in \eqref{topologizable}. 

\ref{PBinH4} $\Rightarrow$ \ref{PBinH1} Fix some compact set $K \subseteq B_{X}$. Since $C_{\varphi}$ is topologizable, we can find some compact set $W \subseteq B_{X}$, and $(a_n)_{n\in\N}$ with $a_n>0$ such that, 
\begin{equation}\label{no-so}
\sup_{x\in K} |f(\varphi^n(x))| \leq \frac{1}{a_n} \sup_{x\in W}|f(x)|  \,,
\end{equation}
for all $f\in H(B_X)$ and  $n\in \N$. By \cite[Corollary~10.7 and Theorem~11.4]{mujica},
the set $L = \widehat{W}_{H(B_{X})}$ (recall \eqref{Fhull}) is compact and contains $W$. We see that $\varphi^{n} (K) \subseteq L$ for every $n$. Suppose that this is not the case and take $x_0\in K$ and $n_0\in\N$ so that $\varphi^{n_0}(x_0)\notin L$. Then there is  $f\in H(B_X)$ such that $|f(\varphi^{n_0}(x_0))| > \sup_{y\in W}|f(y)|$, and  there exists $m\in\N$ such that 
\[
\sup_{y \in W} \frac{\vert f(y)\vert^{m}}{\vert f(\varphi^{n_{0}} (x_{0}) ) \vert^{m}} < a_{n_{0}} \,,
\]
and this clearly contradicts \eqref{no-so} (taking $g=f^{m}$). 
\end{proof}

\begin{prop}
	\label{pdd-ume}
	Let $\varphi:B_{X} \to B_{X}$ be holomorphic. If $C_\varphi :H(B_{X}) \to H(B_{X})$ is power bounded, then it is also uniformly mean ergodic.
\end{prop}

\begin{proof}
	From \cite[Proposition~9.16]{mujica} we know that every bounded subset of $H(B_{X})$ is relatively compact, therefore $H(B_{X})$ is semi-Montel and, in particular, semi-reflexive. Then, as a consequence of \cite[p.~917]{BoPaRi11} (see also \cite[Proposition~3.1]{JoSaSe20}) we have that every power bounded operator is uniformly mean ergodic.
\end{proof}

\section{The space of holomorphic functions of bounded type} \label{Hb}

If $X$ and $Y$ are Banach spaces and $U \subseteq X$ and $V \subseteq Y$ are open sets, a function $f : U \to V$ is of bounded type if it sends $U$-bounded sets to $V$-bounded sets. We consider the space
$H_{b}(B_{X})$ of all holomorphic functions $f : B_{X} \to \mathbb{C}$ of bounded type, endowed with the topology $\tau_{b}$ of uniform convergence on $B_{X}$-bounded sets. This is a Fr\'echet space. 

If $\varphi : B_{X} \to B_{X}$ is holomorphic of bounded type, then clearly $C_{\varphi} : H_{b} (B_{X}) \to H_{b} (B_{X})$ is well defined. On the other hand, we observe that $X'\subseteq H_{b} (B_{X})$ because every functional is trivially Fr\'echet differentiable. In fact, $X'$ is a complemented subspace of $H_{b} (B_{X})$, as we explain below in the proof of Proposition~\ref{PbdNoMeHBB}.  So, if the composition operator is well defined (as a self map on $H_{b} (B_{X})$), then the 
argument in Remark~\ref{hubbard} shows that $\varphi$ has to be holomorphic. Furthermore, \cite[Proposition~3]{GoGu96} shows that $\varphi$ is of bounded type.

Our first goal in this section is to characterise the power boundedness of composition operators on $H_{b}(B_{X})$. As in Remark~\ref{santsaens}, if the composition operator $C_{\varphi}$ is topologizable, then for every $B_{X}$-bounded set $U$ there is some $B_{X}$-bounded set $V$ such that
\[
\sup_{x \in U} \vert f( \varphi^{n} (x) ) \vert \leq \sup_{x \in V} \vert f(x) \vert \,,
\]
and  $C_{\varphi}$ is power bounded. We go further.

\begin{lema}\label{HolomorphicallyConvexBounded}
	If $U$ is an absolutely convex open set on a Banach space $X$, then $\widehat{A}_{H_b(U)}$ is $U$-bounded for every $U$-bounded set $A$.
\end{lema}
\begin{proof}
The polar set of $A$ is a subset of $X'$ and it is contained in $H_{b}(U)$. Then a straightforward computation shows that $\widehat{A}_{H_b(U)}$ is contained in the bipolar of $A$, which by the  Bipolar Theorem coincides with $\overline{\co}(A)$ (the closure of the absolutely convex hull  of $A$). Since $U$ is absolutely convex, \cite[Remark, p.~527]{burlandymoraes} gives that $\overline{\co}(A)$ is $U$-bounded, which completes the proof.
\end{proof}

With exactly the same proof as in Theorem~\ref{PBinH}, replacing `compact' by `$B_{X}$-bounded' we have the following.

\begin{teo}\label{pbddHb}
Let $\varphi:B_{X} \to B_{X}$ be a holomorphic mapping. The following assertions are equivalent
\begin{enumerate}%[(a)]
	\item\label{PBinHb1} $\varphi$ has $B_{X}$-stable orbits.
	\item \label{PBinHb2} $C_\varphi :H_{b}(B_{X}) \to H_{b}(B_{X})$ is power bounded.
	\item \label{PBinHb3} $\big(\frac{1}{n} C_{\varphi}^{n} \big)_{n}$ is equicontinuous in $L(H_{b}(B_{X}))$.
	\item \label{PBinHb4} $C_{\varphi} :H_{b}(B_{X}) \to H_{b}(B_{X})$ is topologizable.		
\end{enumerate}
\end{teo}

We now show that in this case every mean ergodic composition operator is power bounded, and there are power bounded operators that are not mean ergodic.

\begin{prop}\label{PbdMeHBB}
Let $\varphi:B_{X} \to  B_{X}$ a holomorphic mapping. If $C_\varphi:H_b(B_{X}) \to  H_b(B_{X})$ is mean ergodic, then $C_\varphi$ is power bounded.
\end{prop}
\begin{proof}
The mean ergodicity immediately gives that  the sequence $(\frac{1}{n}C_\varphi^n)$ tends to zero (pointwise), so it is pointwise bounded. Since $H_b(B_{X})$ is barrelled (because it is a Fréchet space), it is also equicontinuous on $H_b(B_{X})$. This, in view of Theorem~\ref{pbddHb}, gives the conclusion.
\end{proof}

We want to find now composition operators that are power bounded but not mean ergodic. The shifts defined in  \eqref{def:shifts} provide us with such examples.

\begin{prop}\label{PbdNoMeHBB}
The composition operators $C_{B} :H_b(B_{c_0}) \to  H_b(B_{c_0})$ and $C_F:H_b(B_{\ell_{1}}) \to  H_b(B_{\ell_{1}})$ are power bounded but not mean ergodic.
\end{prop}
\begin{proof}
We already noted in Example~\ref{shifts} that $B$ has $B_{c_{0}}$-stable orbits which, in view of Theorem~\ref{pbddHb}, shows that $C_{B}$ is power bounded. 

We now see that  $C_{B}$ is not mean ergodic. We begin by observing that  $H_{b}(B_{X})$ contains a complemented copy of $X'$ for every Banach space $X$. Indeed, given a holomorphic $f: B_{X} \to \mathbb{C}$, we denote its differential at $0$ (that 
belongs to $X'$) by $df(0)$. Then a simple computation shows that the mappings $P: H_{b}(B_{X}) \to X'$ and $J: X' \to H_{b}(B_{X})$ defined by $P(f) = df(0)$ and $J(u) = \left.u\right|_{B_{X}}$  give our 
claim.

We consider now the restriction of $C_{B}$ to $J(\ell_{1})$ (recall that $c_{0}'= \ell_{1}$) and we have, for each $u \in \ell_{1}$ and $x \in c_{0}$,
\begin{multline*}
	\langle C_B u, x \rangle = u(B (x))= \langle u, B(x) \rangle 
	=\langle (u_1,u_2,u_3,\dots),(x_2,x_3,x_4,\dots) \rangle \\ 
	= u_1x_2+u_2x_3+u_3x_4+\cdots 
	=\langle (0,u_1,u_2,u_3,\dots),(x_1,x_2,x_3,x_4,\dots) \rangle = \langle Fu,x \rangle \,.
\end{multline*}
Thus $F=P \circ C_{B}  \circ J$, which is not mean ergodic on $\ell_{1}$ (see, for instance, \cite{BoPaRi11}). This implies that $C_{B}$ is not mean ergodic on $H_{b}(B_{c_{0}})$. 

For the forward shift, Example~\ref{shifts} showed that $F$ has $B_{\ell_{1}}$-stable orbits. Essentially the same argument as before shows that the restriction of $C_{F}$ to $\ell_{1}'= \ell_{\infty}$ is the backward shift $B$, which is not mean ergodic. This yields the conclusion.
\end{proof}

We look now for sufficient conditions for a given power bounded composition operator to be mean ergodic (and up to some point, even to reverse the implication in Proposition~\ref{PbdMeHBB}). Before we need the following lemma. The argument of the proof is essentially the one in \cite[Chapter~2, Theorem~1.1]{krengel85} (see also \cite[page~908]{BoPaRi11}); since our setting is slightly different we sketch the proof here for the sake of completeness. 

\begin{lema}\label{5}
Let $E$ be a locally convex Hausdorff space, $T \in L(E)$ be power bounded and $x \in E$. If $y \in E$ is a  $\sigma(E,E')$-cluster point of  $(T_{[n]}x)_{n\in\N}$ , then $\lim_{n\to\infty} T_{[n]} x=y$.
\end{lema}
\begin{proof}
Fix $p \in \Gamma_{E}$. Since $T$ is power bounded we can find $q \in \Gamma_{E}$ so that 
$p(T^{n}z) \leq q(z)$
for every $z \in E$. If $y$ is a  $\sigma(E,E')$-cluster point of  $(T_{[n]}x)_{n\in\N}$ then it belongs to the $\sigma(E,E')$-closure of the set which (note that $(T_{[n]}x)_{n\in\N} \subseteq \co (T^{n}x)_{n\in\N_{0}}$) is contained in the $\sigma(E,E')$-closure of $\co (T^{n}x)_{n\in\N_0}$. But as a consequence of the Hahn-Banach theorem, for  convex sets the $\sigma(E,E')$-closure coincides with the closure. So, $y \in \overline{\co}(T^{n}x)_{n\in\N_{0}}$, and for a given $\varepsilon >0$ we can find $z \in \co (T^{n}x)_{n\in\N_{0}}$ so that $q(z-y) <  \varepsilon$, so that we can write
\begin{equation} \label{corona}
p (y - T_{[n]} x ) \leq p(y-T_{[n]} z )  + p (T_{[n]} z - T_{[n]} x ) \,.
\end{equation} 
Note that $z = \sum_{k=0}^{m} \lambda_{k} T^{k} x$ for some $0 \leq \lambda_{k} \leq 1$ with $\sum_{k=0}^{m} \lambda_{k}=1$. We define $S= \sum_{k=0}^{m} \lambda_{k} T^{k}  \in L(E)$. Hence
\[
p( T_{[n]} Sx  - T_{[n]}x ) \leq \sum_{k=0}^{m} \lambda_{k} p( T_{[n]}T^{k} x - T_{[n]}x ) \,.
\]
If $n \geq m \geq k$, we have 
\begin{equation} \label{oud}
p( T_{[n]}T^{k} x - T_{[n]}x ) \leq \frac{1}{n} \Big(  \sum_{j=0}^{k-1} p( T^{j} x ) + \sum_{j=n}^{n+k-1} p( T^{j} x ) \Big) \leq \frac{2k}{n} q(x).
\end{equation}
With this we can estimate the second addend  in the right-hand term of \eqref{corona}. In order to control the first one it is enough to see that $y=Ty$ since, if this is the case then $y=T_{[n]}y$ and $p(y-T_{[n]} z ) \leq q(y-z)$. Given $x' \in E'$, we have
\begin{multline*}
|\langle y-Ty,x'\rangle|
\leq |\langle y- T_{[m]} x,x'\rangle|+ |\langle T_{[m]} x- T T_{[m]} x,x' \rangle| + |\langle T y - T T_{[m]} x,x' \rangle| \\
\leq |\langle y- T_{[m]} x,x'\rangle|+ |\langle T_{[m]} x- T T_{[m]} x,x' \rangle| + |\langle  y -  T_{[m]} x,T' x' \rangle| \,.
\end{multline*}
Since $y$ is a $\sigma(E,E')$-cluster point, we can choose $m$ so that the first and third term are arbitrarily small.
On the other hand, \eqref{oud} implies that $T_{[m]} x- T T_{[m]} x$ tends to $0$ as $m \to \infty$ and, then, so also does the second term. This shows that $y=Ty$ and completes the proof.
\end{proof}

We denote by $\mathcal{P}(X)$ the algebra of all continuous polynomials on $X$ (these are finite sums of homogeneous polynomials), and by $\sigma(X, \mathcal{P}(X))$ the coarsest topology making all $P \in \mathcal{P}(X)$ continuous. This is a Hausdorff topology satisfying $\Vert \cdot \Vert \succeq \sigma(X, \mathcal{P}(X)) \succeq \sigma(X, X^{*})$, and  the concepts of (relatively) countably compact subset, (relatively) sequentially compact subset and (relatively) compact subset all agree with respect to
this topology \cite{galindo2000weakly}.

\begin{prop} \label{662}
Let $\varphi : B_{X} \to B_{X}$ be holomorphic, having $B_{X}$-stable orbits and such that $\varphi (A)$ is relatively $\sigma(X, \mathcal{P}(X))$-compact for every $B_{X}$-bounded set $A$. Then $C_{\varphi} : H_{b} (B_{X}) \to H_{b} (B_{X})$ is mean ergodic.
\end{prop}
\begin{proof}
From Theorem~\ref{pbddHb} we have that the composition operator $C_{\varphi}$ is power bounded, and the equicontinuity of $(C_\varphi^n)_{n\in\N}$ gives that the set $\{(C_\varphi)_{[n]}(f):n\in\N\}$ is bounded for every $f\in H_b(B_{X})$. Now, by \cite[Theorem~2.9]{galindolourencomoraes} $C_\varphi$ maps bounded sets of $H_{b}(B_{X})$ into relatively $\sigma(H_b(B_{X}),H_b(B_{X})')$-compact sets, so for every $f\in H_b(B_{X})$ the set 
\[
C_\varphi \big( \{(C_\varphi)_{[n]}(f) \}_{n} \big) = \Big\{ \frac{1}{n} \sum_{k=1}^{n} C_\varphi^k (f): n\in\N \Big\}
\] is relatively $\sigma(H_b(B_{X}),H_b(B_{X})')$-compact and, therefore it has a  $\sigma(H_b(B_{X}),H_b(B_{X})')$-cluster point. Our aim now is to see that  $\{(C_\varphi)_{[n]}(f) \}_{n}$ has a $\sigma(H_b(B_{X}),H_b(B_{X})')$-cluster point which, using Lemma~\ref{5}, implies that the sequence $((C_\varphi)_{[n]}(f))_{n\in\N}$ converges in $H_b(B_{X})$ for all $f\in H_b(B_{X})$, and $C_\varphi$ is mean ergodic.

Note that 
\[ 
(C_\varphi)_{[n]}(f) = \frac{1}{n}(f - C_\varphi^n (f)) + \frac{1}{n} \sum_{k=1}^{n} C_\varphi^k (f)
\]
for every $n$. The fact that $C_{\varphi}$ is power bounded implies that $\big(\frac{1}{n}(\id (f) -C_\varphi^n (f))\big)_{n\in\N}$ tends to $0$ as $n\to \infty$, and this gives that $\{(C_\varphi)_{[n]}(f) \}_{n}$ has a $\sigma(H_b(B_{X}),H_b(B_{X})')$-cluster point, as we wanted.
\end{proof}

\begin{coro}\label{monteverdi}
Let $X$ be a Banach space such that every $B_{X}$-bounded set is relatively $\sigma(X, \mathcal{P}(X))$-compact. Then $C_\varphi : H_{b}(B_{X}) \to H_{b}(B_{X})$ is power bounded if and only if $C_\varphi$ is mean ergodic.		
\end{coro} 
An example of a Banach space which satisfies such a property is the Tsirelson space $T^*$: it is known that $T^*$ is reflexive and the polynomials on $T^*$ are weakly sequentially continuous \cite[p.~121]{Di99}. Hence, any sequence in the unit ball of $T^*$ has a weakly convergent subsequence, which converges in the topology $\sigma(T^*,\mathcal{P}(T^*)).$ Since $\sigma(T^*,\mathcal{P}(T^*))$ is angelic \cite[p.~150]{galindo2000weakly}, the unit ball is also relatively $\sigma(T^*,\mathcal{P}(T^*))$-compact.

We find now conditions to ensure that a given composition operator is uniformly mean ergodic. Here $C_{0}$ denotes the composition operator defined by the constant  function $0$ (i.e. $C_{0}(f) = f(0)$ for every $f$).

\begin{teo} \label{janacek}
Let $\varphi:B_{X} \to  B_{X}$ be holomorphic so that for every $0<t<1$ there exists $0<\rho<t$ such that
\begin{equation} \label{boulanger}
\varphi(tB_{X})\subseteq \rho B_{X}\,. 
\end{equation}
Then 
\[ 
 C_{{\varphi}^n} \to C_0, 
 \]
in the topology of bounded convergence on $H_b(B_{X})$. In particular,
\begin{equation} \label{giotto}
(C_{\varphi})_{[n]} \to C_0, 
\end{equation}
in the topology of bounded convergence on $H_b(B_{X})$ and $C_{\varphi} : H_{b} (B_{X}) \to H_{b}(B_{X})$ is uniformly mean ergodic.
\end{teo}
\begin{proof}
Fix some $0 < t <1$. First of all, \eqref{boulanger} implies, on the one hand, that $\varphi^n(tB_X)\subset \rho B_X$ for every $n\in \N$ and, on the other hand, that $\varphi(0)=0$. We can then apply Lemma~\ref{schwarz} to the function
$[x \rightsquigarrow \frac{1}{\rho} \varphi (tx) ]$ and get
\begin{equation} \label{rachmaninov}
\Vert \varphi^{n} (x) \Vert \leq \Big( \frac{\rho}{t} \Big)^{n} \Vert x \Vert,
\end{equation}
for every $x \in t B_{X}$ and $n \in \mathbb{N}$. Now, given $f \in H_{b}(B_{X})$, we obviously have $\|f\circ\varphi^n\|_{tB_{X}}\leq \|f\|_{tB_{X}}$ for every $n\in \N$. We define $g : B_{X} \to \mathbb{D}$ by $g(x) = \frac{1}{2 \Vert f \Vert_{tB_{X}}}\big(f ( \varphi(tx)) - f(0)\big)$. This is clearly holomorphic and satisfies $g(0)=0$. Then we can apply Lemma~\ref{schwarz} to $g$ and  \eqref{rachmaninov} to obtain
\[
\Vert C_{{\varphi}}^n (f) - C_0(f) \Vert_{tB_{X}}
= \sup_{x\in tB_{X}} |f({\varphi}^n(x))-f(0)| 
\leq 2\Vert f\Vert_{tB_{X}} \sup_{x\in tB_{X}} \Vert {\varphi}^{n-1}(x) \Vert
\leq 2\Vert f\Vert_{tB_{X}} \Big( \frac{\rho}{t} \Big)^{n-1} \,.
\]
This implies, for every $0<t<1$ and every bounded set $A \subseteq H_{b}(B_{X})$,
\[
\lim_{n\to \infty} \sup_{f\in A} \sup_{x\in tB_{X}} |C_{{\varphi}^n}(f)(x) - f(0)| =0.
\]
 Hence, $C_{{\varphi}^n} \to C_0$ in the topology of bounded convergence.
Once we have this, \eqref{giotto} is a straightforward consequence.
\end{proof}

\begin{nota}\label{squared in l_2}
If $\varphi:B_{X} \to  B_{X}$ is holomorphic and satisfies
\begin{equation} \label{dvorak}
\varphi(B_{X})\subseteq rB_{X} \, \text{ for some } \, 0<r <1  \, \text{ and } \, \varphi (0) =0 \,,
\end{equation}
then, applying Lemma~\ref{schwarz} to the function $[x \rightsquigarrow \frac{1}{r} \varphi (x) ]$ we get $\Vert \varphi (x) \Vert \leq r \Vert x \Vert$ for every $x \in B_{X}$, and this implies that $\varphi$ satisfies \eqref{boulanger} with $\rho=tr$. 

There are, however, functions satisfying \eqref{boulanger} but not \eqref{dvorak}. To see this just consider the restriction to $B_{X}$ of any $m$-homogeneous polynomial (for $m>1$) $P: X \to X$ with $\Vert P \Vert \leq 1$. For a fixed $0<t<1$ take any $0 < \varepsilon < t - t^{m}$ and note that
\[ 
\| P(tx)\| \leq t^{m} \Vert x \Vert^{m} \leq  ( t- \varepsilon) \Vert x \Vert,
\]
for every $x \in B_{X}$. That is, every homogeneous polynomial with norm $\leq 1$ satisfies \eqref{boulanger}. If $\Vert P \Vert=1$ and attains its norm (that is, there is $x_{0}$ with $\Vert x_{0} \Vert =1$ so that $\Vert P(x_{0} ) \Vert = \Vert P \Vert$) then
\[
\big\Vert P \big( (1- \tfrac{1}{n}) x_{0} \big) \big\Vert = \Big( 1 - \frac{1}{n} \Big)^{m} \,,
\]
and there is no $0<r<1$ so that $P(B_{X}) \subseteq r B_{X}$. For a concrete example of such a polynomial just consider the $2$-homogeneous one $P : \ell_{2} \to \ell_{2}$ given by $P\big( (x_{n})_{n} \big) = (x_{n}^{2})_{n}$ (in this case one can take $x_{0} = e_{1}$).

In particular, we have that, if $m>1$ and $P$ is an $m$-homogeneous polynomial with $\Vert P \Vert \leq 1$, then $C_{P} : H_{b} (B_{X}) \to H_{b}(B_{X})$ is uniformly mean ergodic. For $m=1$, that is, for linear operators, this property does not hold,
as Proposition~\ref{beethoven} shows.
\end{nota}

One may also ask if in \eqref{dvorak} we can drop the condition on the fixed point and still get \eqref{boulanger} just assuming that $\varphi(B_{X})\subseteq rB_{X}$ for some  $0<r <1$ . But this is not 
the case: fix some $x_{0} \in B_{X}$ and consider the constant function $\varphi(x) = x_{0}$ for every $x \in B_{X}$.

\subsection{Example of a composition operator which is mean ergodic but not uniformly mean ergodic in $H_b(B_{c_0})$}

The following result is well known \cite[\S39,4(1), p.~138]{K2}.
\begin{lema}\label{lemma 2.1 david}
Let $(T_n)_n$ be a sequence of equicontinuous operators on a locally convex space $E$. If $(T_n)$ is pointwise convergent to a continuous operator $T$ on some dense set $D\subseteq E$, then $(T_n)_n$ is pointwise convergent to $T$ in $E$.
\end{lema}

We also need the following property \cite[Theorem~15.60]{libro_2019}.
\begin{teo}\label{15.58 llibre}
For each $m\in\N$, the set $A_m:=\{x^\alpha: |\alpha|=m\}$ of monomials generates a dense subspace of $\mathcal{P}(^m c_0)$.
\end{teo}

\begin{nota}\label{densitat}
Since $B_{c_0}$ is a balanced set, the polynomials are dense on $H_b(B_{c_0})$. Therefore, by Theorem~\ref{15.58 llibre} we have that the set 
\[ 
\spn \{x^\alpha: \alpha\in \N_0^{(\N)} \} 
\]
is dense on $H_b(B_{c_0})$.
\end{nota}

\begin{prop} \label{beethoven}
Let $F: B_{c_0} \to B_{c_0}$ be the forward shift. The composition operator $C_F:H_b(B_{c_0})\rightarrow H_b(B_{c_0})$ is mean ergodic but not uniformly mean ergodic.
\end{prop}

\begin{proof}
First, we see that $C_F$ is mean ergodic. We follow a similar scheme to that in \cite[Theorem~2.2]{BGJJ2016m}, using that $\spn \{x^\alpha: \alpha\in \N_0^{(\N)} \}$ is a dense subspace of $H_b(B_{c_0})$ (Remark~\ref{densitat}) together with Lemma~\ref{lemma 2.1 david}.  Since $C_F$ is power bounded on $H_b(B_{c_0})$ (because it has $B_{c_0}$-stable orbits), $(C_F^n)_n$ is equicontinuous. Therefore, $\left((C_F)_{[n]}\right)_n$ is also equicontinuous on $H_b(B_{c_0})$. 
Since $C_F(1)=1=C_0(1)$ for any constant mapping (this is in fact true for any composition operator), it remains  to see that $\left((C_F)_{[n]}(h)\right)_n$ $\tau_{b}$-converges to $C_0(h)$ for every  $h\in A_m$ and $m>0$ (on these cases $C_0(h)=0$). 
For $h(x)=x^\alpha$ with $|\alpha|=m$, we define $n_h=\max \{j\in \N: (\alpha)_j\neq 0\}$ which is a finite number. Observe that $C_F^n(h)=C_F^n(x^{\alpha})= (F^n(x))^{\alpha}=0$ for all $n\geq n_h$, and the claim follows. 

As in the proof of Proposition~\ref{PbdNoMeHBB}, one can see that $P\circ C_F\circ J=B$, where $B:\ell_1 \to \ell_1$ is the backward shift (recall  \eqref{def:shifts}). If $C_F$ were uniformly mean ergodic on $H_b(B_{c_0})$, then $B:\ell_1 \to \ell_1$ would be uniformly mean ergodic, but this is not the case. Indeed, since $B^j x$ tends to $0$ in $\ell_1$ for all $x\in \ell_1$, the only possible value for the limit projection of $\frac{1}{N}\sum_{j=0}^{N-1} B^j$ is $0$. But, for each $N\in\N$, we have 
\[
	\sup_{\|x\|\le 1} \Big\|\frac{1}{N}\sum_{j=0}^{N-1} B^j(x)\Big\|_{\ell_1}\ge \frac{1}{N} \Big\|\sum_{j=0}^{N-1} B^j(e_N)\Big\|_{\ell_1}= \frac{1}{N}\big\|(1,\stackrel{(N)}{\ldots},1,0,\ldots)\big\|_{\ell_1}=1.
\]
And it is not true that 
\[
\lim_{N\to \infty} \Big\|\frac{1}{N} \sum_{j=0}^{N-1} B^j\Big\|=0.
\]
\end{proof}

\subsection{The Hilbert-space case} 
Let us go back to \eqref{dvorak} for a moment. If we only assume $\varphi(B_{X}) \subseteq r B_{X}$, the Earle-Hamilton fixed point theorem \cite{Earle-Hamilton} implies that there exists a unique $a\in B_{X}$ such that $\varphi(a)=a$. It is then natural to ask if this is enough to ensure that the composition operator is uniformly mean ergodic. If we restrict ourselves to Hilbert spaces $H$ we can say something in this respect. We need the following lemma.

\begin{lema} \label{johannchristian}
Let $\varphi : B_{H} \to B_{H}$ be holomorphic so that $C_{\varphi^{n}} \to C_{0}$ in the topology of bounded convergence of $L(H_b(B_{H}))$. Then for every $a \in B_{H}$ the mapping $\overline{\varphi} = \alpha_{a} \circ \varphi \circ \alpha_{a}$ satisfies that $C_{\overline{\varphi}^{n}} \to C_{a}$ in the topology of bounded convergence of $L(H_b(B_{H}))$.
\end{lema}
\begin{proof}
Since both $\varphi$ and $\alpha_{a}$ are of bounded type (see Lemma~\ref{auto rB}), the composition $\alpha_{a} \circ \varphi \circ \alpha_{a}$ is of bounded type and $C_{\overline{\varphi}} : H_{b} (B_{H}) \to H_{b}(B_{H})$ is well defined. Observe  now that $\overline{\varphi}^n=\alpha_a\circ \varphi^n \circ \alpha_a$ for all $n\in \N$ since $\alpha_a^{-1}=\alpha_a$. Then
\[
C_{\overline{\varphi}^{n}} = C_{\alpha_a\circ \varphi^n \circ \alpha_a} = C_{\alpha_a}\circ C_{\varphi^n} \circ C_{\alpha_a}
\to C_{\alpha_a}\circ C_0 \circ C_{\alpha_a}= C_{\alpha_a}\circ C_{\alpha_a(0)}  = C_{\alpha_a}\circ C_{a}=C_{a}.
\]
\end{proof}

\begin{prop} \label{enigma}
Let $\varphi:B_{H} \to  B_{H}$ be holomorphic such that 
\begin{equation} \label{eskenian}
\varphi(B_{H})\subseteq rB_{H} \, \text{ for some } \, 0<r <1\,.
\end{equation} 
Then, for the unique $a\in B_{H}$ such that $\varphi(a)=a$ we have $
 C_{{\varphi}^n} \to C_a$ 
in the topology of bounded convergence of $L(H_b(B_{H}))$. In particular $(C_{\varphi})_{[n]} \to C_a$, and  $C_{\varphi} : H_{b} (B_{H}) \to H_{b}(B_{H})$ is uniformly mean ergodic.
\end{prop}
\begin{proof}
Define $\phi = \alpha_{a} \circ \varphi \circ \alpha_{a} : B_{H} \to B_{H}$, which clearly satisfies $\phi(0)=0$. Also, 
\[ 
\phi(B_{H})=(\alpha_a\circ \varphi \circ\alpha_a)(B_{H})= (\alpha_a\circ \varphi)(B_{H})\subseteq \alpha_a(rB_{H}) \,,
\]
and using Lemma~\ref{auto rB} we can find some $0<\varepsilon<1$ so that
\[ 
\phi(B_{H}) \subseteq (1-\varepsilon) B_{H} \,. 
\]
Then $\phi$ satisfies \eqref{dvorak} and, by Theorem~\ref{janacek}, $C_{\phi^{n}} \to C_{0}$. Since $\varphi = \alpha_{a} \circ \phi \circ \alpha_{a}$ (because $\alpha_{a}^{-1}= \alpha_{a}$), Lemma~\ref{johannchristian} yields the claim.
\end{proof}

Let us consider any analytic self-map $\varphi : B_{X} \to B_{X}$ (being $X$ any Banach space) so that $\varphi \circ \varphi = \id$. Then 
\[
C_{\varphi}^{n} = \begin{cases}
C_{\varphi} \text{ if } n \text{ is odd}, \\
C_{\id_{B_{X}}} = \id_{H_{b}(B_{X})}  \text{ if } n \text{ is even} \,,
\end{cases}
\]
and for each $k \in \mathbb{N}$ we have
\[
(C_{\varphi})_{[2k-1]}=\frac{1}{2k-1}\sum_{n=0}^{2k-1}C_{\varphi}^{n}
= \frac{k}{2k-1} \big( C_{\varphi} +\id_{H_{b}(B_{X})}  \big),
\]
and
\[
(C_{\varphi})_{[2k]}=\frac{1}{2k}\sum_{n=0}^{2k}C_{\varphi}^{n}
= \frac{1}{2} \big( C_{\varphi} +\id_{H_{b}(B_{X})}  \big)+ \frac{1}{2k}\id_{H_{b}(B_{X})}\,.
\]
This implies that $\lim_{n\to\infty}(C_{\varphi})_{[n]}=  \frac{1}{2} \big( C_{\varphi} +\id_{H_{b}(B_{X})}  \big)$ in the topology of bounded convergence  of $L(H_b(B_{X}))$, and 
$C_{\varphi} : H_{b} (B_{X}) \to H_{b}(B_{X})$ is uniformly mean ergodic. Note that $\alpha_{a} : B_{H} \to B_{H}$ (now $H$ being a Hilbert space) satisfies this condition, so that $C_{\alpha_{a}} : H_{b} (B_{H}) \to H_{b}(B_{H})$ is uniformly mean ergodic. However, $\alpha_{a}$ does not satisfy neither \eqref{boulanger} nor \eqref{eskenian}.

\section{The space of bounded holomorphic functions} \label{Hi}

We consider now the space $H^{\infty} (B_{X})$ of all holomorphic functions $f : B_{X} \to \mathbb{C}$ that are bounded. With the norm $\Vert f \Vert_{\infty} = \sup_{x \in B_{X}} \vert f(x) \vert$ it becomes a Banach space. We look at composition operators $C_{\varphi} : H^{\infty} (B_{X}) \to H^{\infty} (B_{X})$. If $\varphi : B_{X} \to B_{X}$, then
\[ 
\|C_\varphi^n(f)\|_{\infty} = \sup_{x\in B_{X}}|C_\varphi^n(f)(x)| 
= \sup_{x\in B_{X}} |f(\varphi^n(x))| 
\leq \sup_{x\in B_{X}} |f(x)| = \|f\|_{\infty} \,,
\]
and $\Vert C_{\varphi}^{n} \Vert \leq 1$ for all $n\in\N$. Hence every $C_{\varphi}$ that is well defined on $H^{\infty} (B_{X})$ is power bounded.  Since $(X', \|\cdot\| ) = (X', \tau_b )$,  the dual space $X'$ is also complemented in $H^{\infty} (B_{X})$, and the same arguments as in Proposition~\ref{PbdNoMeHBB} give examples of 
composition operators $C_{\varphi} : H^{\infty} (B_{X}) \to H^{\infty} (B_{X})$ which are not mean ergodic. However, $X'$ is in general not complemented in $H(B_{X})$ since $(X', \|\cdot\| ) \neq (X', \tau_{0} )$ and  these arguments do not work for $H(B_{X})$.

We give now conditions on the symbol to define a uniformly mean ergodic composition operator on $H^{\infty} (B_{X})$. 

\begin{prop} \label{servicio}
Let $\varphi:B_{X} \to  B_{X}$ be holomorphic such that $\varphi(B_{X})\subseteq rB_{X}$ for some $0<r <1$  and $\varphi (0) =0$. Then 
\[ 
 C_{{\varphi}^n} \to C_0, 
 \]
in the norm operator topology of $L(H^{\infty}(B_{X}))$. In particular, $(C_{\varphi})_{[n]} \to C_0$ , and $C_{\varphi} : H^{\infty} (B_{X}) \to H^{\infty} (B_{X})$ is uniformly mean ergodic.
\end{prop}
\begin{proof}
Take some $f \in H^{\infty} (B_{X})$ with $\Vert f \Vert_{\infty} \leq 1$. Defining $g: B_{X} \to \mathbb{D}$ by $g(x) = \frac{1}{2} (f(x) - f(0))$ and using Lemma~\ref{schwarz} we get
\[
\vert f(x) - f(0) \vert \leq 2 \Vert x \Vert
\]
for every $x \in B_{X}$. Proceeding as in \eqref{rachmaninov}, we get that $\Vert \varphi^{n} (x) \Vert \leq r^{n} \Vert x \Vert$
for every $x \in B_{X}$ and $n \in \mathbb{N}$. This yields
\[
\big\vert  f \big(\varphi^{n} (x )\big) - f(0) \big\vert \leq 2 \Vert \varphi^{n} (x) \Vert \leq 2 r^{n} \Vert x \Vert \,.
\]
Therefore
\[
\Vert C_{\varphi}^n - C_0 \Vert_{L(H^\infty(B))} 
= \sup_{\|f\|_{\infty}\leq 1} \, \sup_{x\in B_{X}} \big\vert  f \big(\varphi^{n} (x )\big) - f(0) \big\vert 
\leq 2 \sup_{x\in B_{X}}  \Vert \varphi^{n} (x) \Vert \leq 2 r^{n} \,,
\]
which gives the claim.
\end{proof}

We observe  that the hypothesis in Proposition~\ref{servicio} is exactly the same one as \eqref{dvorak} in Remark~\ref{squared in l_2}. One can ask if the result also holds assuming instead \eqref{boulanger}. This is not the case. We already saw in Remark~\ref{squared in l_2} that the mapping $P : B_{\ell_{2}} \to B_{\ell_{2}}$ given by $P\big( (x_{n})_{n} \big) = (x_{n}^{2})_{n}$ satisfies \eqref{boulanger}. Then, by Theorem~\ref{janacek} the Ces\`{a}ro means of $C_{P}$ converge to $C_0$. Hence, $C_{P}: H_{b}(B_{\ell_{2}} ) \to H_{b} (B_{\ell_{2}})$ is uniformly mean ergodic.

However, the operator $C_{P} : H^{\infty} (B_{\ell_{2}}) \to  H^{\infty} (B_{\ell_{2}})$ is not even mean ergodic. Notice that $H^\infty(B_{\ell_{2}})\subseteq H_b(B_{\ell_{2}})$ and $\tau_{b}$ is weaker than the norm topology. Then if $C_{P}$ were 
mean ergodic,  $\big( (C_{P})_{[n]}(f) \big)_{n}$ should converge in norm to $C_{0}(f)$ for every $f \in H^{\infty} (B_{\ell_{2}})$. Take $f \in H^{\infty}(B_{\ell_{2}})$ given by $f\big((x_{n})_{n} \big) = x_{1}$ and consider  $z_{m} =  (1-\frac{1}{m}) e_{1} \in B_{\ell_{2}}$ for each $m \in \mathbb{N}$. Then $P^{k} (z_{m}   ) = (1-\frac{1}{m})^{2^k} e_{1}$ for every $k$ and
\[
(C_{P})_{[n]} (f) ( z_{m} ) - C_{0}(f) (z_{m})
= \frac{1}{n} \sum_{k=0}^{n-1} f \big( (1-\frac{1}{m})^{2^k} e_{1} \big) - f(0) 
=  \frac{1}{n} \sum_{k=0}^{n-1}  \big( 1-\frac{1}{m} \big)^{2^k} \,.
\]
Thus
\[
\sup_{x \in B_{\ell_{2}}} \vert (C_{P})_{[n]} (f) ( x ) - C_{0}(f) (x) \vert
\geq \sup_{m \in \mathbb{N}}  \frac{1}{n} \sum_{k=0}^{n-1}  \big( 1-\frac{1}{m} \big)^{2^k}  =1 \,,
\]
and $\big( (C_{P})_{[n]}(f) \big)_{n}$ does not converge in norm to $C_{0}(f)$. This finally shows that $C_{P} : H^{\infty} (B_{\ell_{2}}) \to H^{\infty} (B_{\ell_{2}})$ is not mean ergodic. 

The same argument as in Lemma~\ref{johannchristian} and Proposition~\ref{enigma} shows the following.

\begin{prop}\label{Convergence in H^inf}
Let $\varphi:B_{H} \to  B_{H}$ be  analytic such that $\varphi(B_{H})\subseteq rB_{H}$ for some $0<r <1$. Then, for the unique $a\in B$ such that $\varphi(a)=a$ we have that $C_{{\varphi}^n} \to C_a$ in the norm of $L(H^{\infty}(B_{H}))$. In particular $(C_{\varphi})_{[n]} \to C_{a}$ and $C_{\varphi}$ is uniformly mean ergodic.
\end{prop}

We have formulated our results for the open unit ball of a Banach (or Hilbert) space, mainly with the purpose of simplicity and to give a uniform presentation of our results. Our proofs, however, transfer with no extra 
effort, to some other, more general settings. Let us briefly point out how.
\begin{itemize}
\item A set $U$ is said to be holomophically convex if $\widehat{K}_{H(U)}$ (recall in \eqref{Fhull}) is compact for every compact set $K \subseteq U$ (see \cite[Definition~11.3]{mujica}). The
proof of Theorem~\ref{PBinH} transfers word by word if $B_{X}$ is replaced by a holomorphically convex set $U$.

\item The proof of Proposition~\ref{fix0} works exactly in the same way if $B_{X}$ is replaced by any absolutely convex open set $U$. Then, Theorem~\ref{pbddHb}, as well as Propositions~\ref{PbdMeHBB}, \ref{662} and Corollay~\ref{monteverdi} also hold for arbitrary absolutely convex open sets (note that \cite[Theorem~2.9]{galindolourencomoraes} also holds in this case).

\item The key element in the proofs of Propositions~\ref{phi-bar-Bso}, \ref{enigma} and~\ref{Convergence in H^inf} (stated for the open unit ball of a Hilbert space) is the existence of a family of biholomorphic automorphisms 
on the ball (as in \eqref{alpha}) satisfying \eqref{castillo}. Hilbert spaces are not the only examples of such a situation. In every $C^{*}$-algebra, for example, also such a family of automorphisms can be defined. In fact, there is a wider class of Banach spaces, known as $JB^{*}$-triples, that also have this property: if $X$ is a $JB^{*}$-triple, then there is a family of biholomorphic automorphisms $\{ \alpha_{a}  \}_{a \in B_{X}}$ on $B_{X}$ satisfying $\alpha_a(0)=a$, $\alpha_a(a)=0$, $\alpha_a^{-1}=\alpha_a$. The class of  $JB^{*}$-triples includes Hilbert spaces and $C^{*}$-algebras, but also wider classes such as $J^{*}$-algebras (closed subspaces of the space of operators between two Hilbert spaces $L(H_{1}, H_{2})$ which are closed under $T \rightsquigarrow T T^{*} T$, being $T^{*}$ the adjoint of $T$); the interested reader may find more information on the subject in \cite{harris1981,mellon2000}. 
Moreover, these automorphisms satisfy the corresponding analogue of Lemma~\ref{auto rB} \cite[Lemma~1]{mackeysevillavallejo2006}. So, the aforementioned results remain valid if $B_{H}$ is replaced by the open unit ball $B_{X}$ of a $JB^{*}$-triple (in particular a $C^{*}$-algebra) $X$.
\end{itemize}

Finally, we observe that some questions remain open. The first one is whether or not \eqref{eskenian} implies that the composition operator $C_{\varphi}$ is uniformly mean ergodic (that is, Proposition~\ref{enigma} extends to arbitrary Banach spaces).
It would also be interesting to find examples of the following situations:
\begin{enumerate}
\item A composition operator on $H(B_{X})$ which is mean ergodic but not uniformly mean ergodic.

\item A composition operator on $H(B_{X})$ which is mean ergodic but not power bounded.

\item A composition operator on $H^{\infty}(B_{X})$ which is mean ergodic but not uniformly mean ergodic.
\end{enumerate}

\textbf{Acknowledgements.} We are very grateful to Jos\'e Bonet for valuable suggestions about this work and to Pablo Galindo for pointing out an example of a Banach space which satisfies Corollary~\ref{monteverdi}. We also thank the referee for her/his careful reading and useful suggestions.
The research of the first author was partially supported by the project MTM2016-76647-P. The research of the second author was partially supported by the project GV Prometeo 2017/102. The research of the third author was supported by the project MTM2017-83262-C2-1-P.
%
%
%\bibliographystyle{abbrv}
%\bibliography{biblio_JSS_hol}

\end{document}